\documentclass[reqno,12pt]{amsart}
 \usepackage{amsmath}
 \usepackage{amsthm}
 \usepackage{amssymb}
 \usepackage{latexsym,longtable}
 \usepackage{graphicx}
 \usepackage{multicol}
 \usepackage{mathrsfs}
 \usepackage{young}
 \usepackage[vcentermath,enableskew,stdtext]{youngtab}
 \usepackage[left=1.1in,right=1.1in,top=1.06in,bottom=1.1in]{geometry}
 \usepackage[all]{xy}
    \SelectTips{cm}{10}     
    \everyxy={<2.5em,0em>:} 
 \usepackage{fancyhdr}      
 \usepackage{multicol}
 \usepackage{multirow}
 \usepackage{enumitem}
 \usepackage{bm}
 \usepackage{stmaryrd}
 \usepackage{etex}

\usepackage{tabmac}
\usepackage{ytableau}
\usepackage{tikz}
\usepackage{tkz-euclide}
\usetikzlibrary{backgrounds}
\usetikzlibrary{arrows}

\usepackage[colorinlistoftodos]{todonotes}

\definecolor{amethyst}{rgb}{0.5, 0.3, 0.7}

 \usepackage{float}
 \usepackage{array}
 \usepackage{colortbl}
 \usepackage{mathtools}
 \restylefloat{figure}
\numberwithin{equation}{section}
 \usetikzlibrary{shapes}
 \usetikzlibrary{arrows}

 \usetikzlibrary{snakes}
\usepackage{url}

\usepackage{amsfonts,epsfig,color}
\makeatletter
\newcommand*{\centerfloat}{%
  \parindent \z@
  \leftskip \z@ \@plus 1fil \@minus \textwidth
  \rightskip\leftskip
  \parfillskip \z@skip}
\makeatother

\theoremstyle{plain}
\newtheorem{theorem}{Theorem}[section]
\newtheorem{lemma}[theorem]{Lemma}

\newtheorem{conjecture}[theorem]{Conjecture}

\theoremstyle{definition}
\newtheorem{definition}[theorem]{Definition}

\newtheorem{example}[theorem]{Example}

\newcommand{\ignore}[1]{}

\newcommand{\NN}{\ensuremath{\mathbb{N}}}

\newcommand{\be}{\begin{equation}}
\newcommand{\ee}{\end{equation}}

\usetikzlibrary{positioning}
\usetikzlibrary{arrows.meta}
\tikzstyle{mystealthbig} = [-{Latex[length=2.1mm, width=1.53mm]}]
\tikzstyle{mystealth} = [-{Latex[length=1.8mm, width=1.34mm]}]
\tikzstyle{mystealthsmall} = [-latex]

\tikzstyle{vertex}=[circle, fill, inner sep=0pt, minimum size=4pt, outer sep = .3pt]

\usepackage{ytableau}
\usepackage{relsize}
\usepackage{stackengine}
\usepackage{hyperref}

\newcommand{\setmid}{\,\,\vrule\,\,}

\newcommand{\mathhuger}[1]{\mathlarger{\mathlarger{\mathlarger{#1}}}}

\newcommand\restr[2]{{
  \left.\kern-\nulldelimiterspace 
  #1 
  \littletaller 
  \right|_{#2} 
  }}

\newcommand{\littletaller}{\mathchoice{\vphantom{\big|}}{}{}{}}

\author{Holden Eriksson}
\address{Department of Mathematics, University of Virginia,
Charlottesville, VA 22903, USA}
\email{mmw2va@virginia.edu}

\title{A note on a cyclotomic-friendly application of RSK}

\begin{document}

\begin{abstract}
We give a combinatorial realization of a level-$\ell$ Robinson–Schensted–Knuth correspondence conjectured to exist by Song and Wang for cyclotomic Schur categories. We show that cyclotomic basis elements can be canonically reorganized into flagged block composition matrices encoding families of biwords, so that the correspondence is obtained by applying the classical RSK correspondence componentwise. This perspective identifies the level-$\ell$ correspondence as an iteration of classical RSK, specializing to the usual correspondence when $\ell=1$ and behaving naturally under restriction to lower levels.
\end{abstract}
\maketitle

\section{Introduction}

The Robinson--Schensted--Knuth correspondence is a bijection between (bi)words or nonnegative integer matrices and pairs of Young tableaux of the same shape, with applications to symmetric functions, permutation statistics such as longest increasing subsequences, and the representation theory of the symmetric and general linear groups. With the development of web categories, representation theories traditionally governed by Young tableaux have been reformulated in diagrammatic terms that incorporate additional structure, such as multiple components, block decompositions, and restrictions on how labels may appear. In this setting, the relevant combinatorial data consist not of a single word or matrix but of several interrelated pieces, and the resulting tableau data record this multi-component structure. This viewpoint naturally raises the question of whether classical correspondences such as RSK admit higher--level analogues that organize such data in a way compatible with the additional structure.

In \cite{beeo}, the authors introduced the \emph{Schur category}, whose path algebra recovers certain subalgebras of Schur algebras, alongside a category \emph{Web}, a $\mathfrak{gl}_{\infty}$--analogue of the $\mathfrak{sl}_n$ web categories. They showed that these categories are isomorphic, and that both contain the (degenerate) Hecke category via fully faithful functors. Building on this framework, Song and Wang \cite{wang1,wang2} defined \emph{affine Schur} and \emph{affine web} categories. In contrast to the finite case, the affine web category embeds as a full subcategory of the affine Schur category rather than being equivalent to it. Their work further introduced \emph{cyclotomic} quotients of these categories, which retain embeddings of the cyclotomic Hecke category and exhibit rich combinatorial structure.

A key result of \cite{wang2} identifies the path algebra of the cyclotomic Schur category with the cyclotomic Schur algebra of Dipper--James--Mathas \cite{djm}. Song and Wang describe two natural bases for this algebra. The first is indexed by pairs of tableaux and corresponds, under this identification, to the cellular basis of the cyclotomic Schur algebra. The second is a family of basis elements indexed by enriched block matrices arising from elementary ``chicken--foot'' diagrams in the cyclotomic web category. The coexistence of these two parametrizations naturally suggests a bijection relating block-matrix data to pairs of tableaux in the cyclotomic setting.

\begin{conjecture}[Song--Wang {\cite{wang2}}]\label{conj}
There exists, for each level $\ell$, a combinatorial bijection
\begin{equation}
\label{e:psi}
\psi_\ell : \mathrm{ParMat}_{\nu,\mu}^{\flat} \longrightarrow 
\bigsqcup_{\lambda \vdash_\ell n}
\mathrm{SST}_\ell(\lambda,\mu)\times \mathrm{SST}_\ell(\lambda,\nu),
\end{equation}
which extends the level $\ell-1$ bijection and recovers the classical RSK correspondence when $\ell=1$.
\end{conjecture}

We confirm this conjecture by giving an explicit bijection at level~$\ell$ that uses only the classical RSK correspondence. The key step is to recast the combinatorial parameters in the set $\mathrm{ParMat}^{\flat}$  as \emph{flagged block composition matrices}. These matrices encode, in a transparent way, collections of biwords subject to the level-dependent labeling restrictions. With this reformulation in hand, the bijection is obtained by applying the classical RSK correspondence independently in each component and then assembling the outputs as semistandard multitableaux of common shape.

Our approach proceeds in three steps. First, we reinterpret the $\mathrm{ParMat}^{\flat}$-indexed data in terms of the {flagged block composition matrices}, which present the same information in a form amenable to combinatorial manipulation. Second, we establish a natural bijection between these matrices and collections of \emph{flagged biwords}, extending the classical correspondence between nonnegative integer matrices and biwords. Finally, we apply the classical RSK correspondence componentwise (after a simple relabeling of the alphabets) to obtain pairs of semistandard multitableaux of common shape.

The resulting correspondence is compatible with restriction to lower levels, specializes to the usual RSK correspondence when $\ell=1$, and reflects the multipartition data inherent in the cyclotomic setting. In particular, it provides a direct combinatorial explanation for the relationship between the tableau--indexed and diagrammatic bases of the cyclotomic Schur category.

\noindent {\bf Acknowledgement.} We thank Linliang Song and Weiqiang Wang for formulating the initiating conjecture. We are especially grateful to Weiqiang Wang for personally recommending this combinatorial problem.

\section{Definitions}

\subsection{Classical RSK ingredients}

Set  $\NN\coloneq\{1,2,3,\ldots\}$. We use {\it composition} to refer to a vector $\alpha=(\alpha_1,\alpha_2,\ldots,\alpha_k)$ of nonnegative integers. A composition is denoted $\alpha\vDash n$ where $n\coloneq\sum_i \alpha_i$. If a composition $\lambda\vDash n$ satisfies $\lambda_i\geq\lambda_{i+1}>0$ for all $i$, then we call it a {\it partition} and denote this fact by $\lambda\vdash n$. The \textit{Ferrers diagram} of shape $\lambda\vdash n$ is given by
\[
  F(\lambda)\coloneq\{(i,j) \setmid i,j\geq1\text{ and }j\leq\lambda_i\}.
  \]
Given a totally ordered set $\mathcal{A}$, a \textit{tableau $T$ of shape $\lambda$} is a filling of the Ferrers diagram with letters from $\mathcal{A}$ (precisely, it is map $T:F(\lambda)\rightarrow\mathcal{A}$). We say $T$ is \textit{semistandard} if it has weakly increasing rows and strictly increasing columns (precisely, if $T(i,j)<T(i+1,j)$ and $T(i,j)\leq T(i,j+1)$ hold for all $i,j$). Let $\mathrm{SST}(\mathcal{A},\lambda)$ denote the set of all semistandard tableaux of shape $\lambda$.
When \(\mathcal A=\mathbb N:=\{1,2,3,\ldots\}\), a tableau \(T\) can be attached to a composition \(\mu \vDash n\) called its \textit{content}, defined by
\[
\mu_i = \#\{\text{entries of }T\text{ equal to } i\}=|T^{-1}(i)|.
\]
We write \(\mathrm{wt}(T)=\mu\), and denote by \(\mathrm{SST}(\mathbb N,\lambda,\mu)\) the set of semistandard tableaux of shape \(\lambda\) and content \(\mu\).

The central object in the RSK correspondence is the biword. Given a totally ordered set of letters $\mathcal{A}$, a \textit{biword of length }$k$ is a $2\times k$ array

\[
w=\begin{pmatrix}
\alpha_1 & \alpha_2 & \cdots & \alpha_k \\
\beta_1 & \beta_2 & \cdots & \beta_k
\end{pmatrix},
\]
with entries in \(\mathcal A\) and
lexicographically increasing columns; that is,
\[
\alpha_1 \le \alpha_2 \le \cdots \le \alpha_k,
\]
and whenever \(\alpha_i=\alpha_{i+1}\), we also have \(\beta_i \le \beta_{i+1}\). We write
\[
w_{\mathrm{top}}=\alpha_1\alpha_2\cdots\alpha_k
\qquad\text{and}\qquad
w_{\mathrm{bot}}=\beta_1\beta_2\cdots\beta_k,
\]
and denote by \(\mathrm{BW}(\mathcal A)\) the set of such biwords.

\begin{theorem}\label{classicalRSK}[Robinson--Schensted--Knuth \cite{RSK}]
There is a bijection
\[
\mathrm{BW}(\mathbb N)
\longleftrightarrow
\coprod_{\lambda\vdash n}
\mathrm{SST}(\mathbb N,\lambda)
\times
\mathrm{SST}(\mathbb N,\lambda).
\]
For biword $w$ corresponding to $(P,Q)$, 
\(
\mathrm{wt}(P)=\text{content of } w_{\mathrm{bot}}\) 
and
\(\mathrm{wt}(Q)=\text{content of } w_{\mathrm{top}}.
\)
\end{theorem}

Another formulation of the RSK correspondence involves the $\mathbb N$-matrix, defined as a map
\[
A:\mathbb N^2\to\mathbb N,
\]
with finite support. We write $A=(a_{ij})$, where $a_{ij}=A(i,j)$.
We say that $A$ has \textit{row sum} $\nu$ and \textit{column sum} $\mu$, writing
$\mathrm{row}(A)=\nu$ and $\mathrm{col}(A)=\mu$, if
\[
\nu_t=\sum_j a_{tj}
\qquad\text{and}\qquad
\mu_t=\sum_i a_{it}.
\]

Each biword $w\in \mathrm{BW}(\mathbb N)$ naturally determines an
$\mathbb N$-matrix $A$ by setting $a_{ij}$ equal to the multiplicity of the
biletter $\begin{pmatrix} i \\ j \end{pmatrix}$ in $w$.
Under this correspondence,
\[
\mathrm{row}(A)_i=\#\{\text{occurrences of } i \text{ in } w_{\mathrm{top}}\}
\quad\text{and}\quad
\mathrm{col}(A)_j=\#\{\text{occurrences of } j \text{ in } w_{\mathrm{bot}}\}.
\]
Thus the RSK correspondence may equivalently be written as a bijection
\begin{equation}
\label{e rskmatrix}
    \left\{\NN\mathrm{-matrices}\,\,A\setmid\begin{tabular}{c}$\mathrm{row}(A)=\nu$\\$\mathrm{col}(A)=\mu$\end{tabular}\right\}\longleftrightarrow\coprod_{\lambda\vdash n}\text{\rm SST}(\NN,\lambda,\mu)\times\text{\rm SST}(\NN,\lambda,\nu).
  \end{equation}

\subsection{Semistandard multitableaux}

We begin by defining the index set for the cellular basis of the cyclotomic Schur algebra. Its elements make up the codomain of $\psi_\ell$ in \eqref{e:psi}


\begin{definition}\label{def:l-multi}
Let $\ell,n\in\mathbb{N}$.
\begin{enumerate}
\item An \emph{$\ell$-multicomposition} of $n$ (resp.\ an \emph{$\ell$-multipartition} of $n$) is an $\ell$-tuple
\[
\mu=(\mu^{(1)},\ldots,\mu^{(\ell)}),
\]
such that each $\mu^{(i)}$ is a composition (resp.\ a partition) and
\(
\sum_{i=1}^{\ell}\lvert \mu^{(i)}\rvert = n.
\)
We write $\mu \vDash_{\ell} n$ (resp.\ $\mu \vdash_{\ell} n$).

\item Let $\mathcal S:=\{\,a_b \mid a,b\in\mathbb{N}\,\}$ be endowed with the fixed total order
\begin{equation}\label{eq:S-order}
1_1<2_1<3_1<\cdots<1_2<2_2<3_2<\cdots<1_3<2_3<\cdots .
\end{equation}
For $\mu\vDash_{\ell} n$, the \emph{$\mu$-alphabet} is the finite multiset $U_\mu\subset S$ in which the letter $a_b$ occurs with multiplicity equal to the $a$-th part of the composition $\mu^{(b)}$ (interpreted as $0$ if that part is absent). Equivalently,
\[
\#\{a_b\in U_\mu\} = \mu^{(b)}_{a}\qquad (a\ge 1,\ 1\le b\le \ell).
\]
\end{enumerate}
\end{definition}

\begin{example}
If $\mu=(23,\ 1,\ 211)$ (so $\ell=3$), then
\[
\mu^{(1)}=(2,3),\quad \mu^{(2)}=(1),\quad \mu^{(3)}=(2,1,1),
\]
and the $\mu$-alphabet is
\[
U_\mu=\{\,1_1,1_1,\;2_1,2_1,2_1,\;1_2,\;1_3,1_3,\;2_3,\;3_3\,\}.
\]
\end{example}

\begin{definition}
Given a multipartition $\lambda\vdash_\ell n$, we say that a \textit{multitableau of shape $\lambda$} is a an $\ell$-tuple $T=\left(T^{(1)},T^{(2)},\ldots,T^{(\ell)}\right)$ such that each $T^{(i)}\in\text{SST}(\mathcal{S},\lambda^{(i)})$. 
Given an $\ell$-multicomposition $\mu\vDash_\ell n$ with associated $\mu$-alphabet $U_\mu\subset S$, we say that $T$ has \emph{content $\mu$} if the multiset of entries appearing in $T^{(1)}\sqcup\cdots\sqcup T^{(\ell)}$ is exactly $U_\mu$. We call $T$ \emph{semistandard} when it satisfies the additional {\it flagging condition}: a letter $a_b$ can occur in $T^{(i)}$ only if $b\ge i$. Let $\text{SST}_\ell(\lambda,\mu)$ denote the set of semistandard multitableaux of shape $\lambda$ and content $\mu$.
\end{definition}

\begin{example}
The multitableau $R=\big(T^{(1)},T^{(2)},T^{(3)}\big)$ below is semistandard with shape $\lambda=(32,2,21)$ and content $\mu=(13,11,22)$
\begin{center}
$
R=
\left(\,
\begin{ytableau}
1_1&2_1&2_1\\
2_1&1_2
\end{ytableau}
\,\mathhuger{,}\,
\begin{ytableau}
2_2&1_3
\end{ytableau}
\,\mathhuger{,}\,
\begin{ytableau}
1_3&2_3\\
2_3
\end{ytableau}
\right)
$.
\end{center}
\end{example}

\subsection{Chicken-foot matrices}

The cyclotomic web category is connected in \cite{wang2} to a combinatorial model called 
(elementary) chicken foot diagrams.  
The diagrammatic parameters indexing the cyclotomic web basis may be encoded by certain partition-enriched block matrices.  These matrices are packaged into a set
\[
\mathrm{ParMat}^{\flat}_{\nu,\mu},
\]
indexed by multicompositions $\nu,\mu \vDash_\ell n$ which record the prescribed row and column data of the underlying block matrix.
Informally, an element of $\mathrm{ParMat}^{\flat}_{\nu,\mu}$ consists of:
\begin{itemize}
\item an $\ell\times\ell$ block matrix of nonnegative integers whose block row and column sums are fixed by $\nu$ and $\mu$, and
\item a choice of a partition in each matrix entry, subject to natural size and length constraints determined by the block position.
\end{itemize}

Precisely, an \textit{$\ell\times\ell$ block $\NN$-matrix} is a collection $A=\{A^{(pq)}\}_{1\leq p,q\leq\ell}$ of $\NN$-matrices. We denote the entry $(i,j)$ of the matrix $A^{(pq)}$ by $a_{ij}^{(pq)}$. The row and column sums of $A$ are the $\ell$-multicompositions $\mathrm{row}(A)$ and $\mathrm{col}(A)$ defined by
\begin{align*}
\text{row}(A)^{(s)}_t=\sum_{q=1}^\ell\sum_{j}a^{(sq)}_{tj}\qquad\text{and}\qquad
\text{col}(A)^{(s)}_t=\sum_{p=1}^\ell\sum_{i}a^{(ps)}_{it}.
\end{align*}

A \textit{partition matrix} is a matrix whose entries are partitions. Equivalently, it is a map $P:\NN^2\rightarrow\mathcal{P}$, where $\mathcal{P}$ is the set of all integer partitions, such that all but finitely many entries are the empty partition. We write $P=(\eta_{ij})$ where $\eta_{ij}$ is the partition in position $(i,j)$.
\begin{definition}
  Fix $\nu,\mu\vDash_\ell n$. The set $\mathrm{ParMat}^{\flat}_{\nu,\mu}$ consists of all pairs $(A,P)$ where:
  \begin{enumerate}
    \item $A=\{A^{(pq)}\}_{1\leq p,q\leq\ell}$ is an $\ell\times\ell$-block $\NN$-matrix.
    \item The block matrix $A$ satisfies $\text{row}(A)=\nu$ and $\text{col}(A)=\mu$.
    \item $P=\{P^{(pq)}\}_{1\leq p,q\leq\ell}$ is a collection of partition matrices $P^{(pq)}$.
    \item For each $(p,q,i,j)$, the largest part of $\eta^{(pq)}_{ij}$ is at most $a^{(pq)}_{ij}$.
    \item For each $(p,q,i,j)$, the partition $\eta^{(pq)}_{ij}$ satisfies $\text{length}(\eta_{ij}^{(pq)})\leq\min(p,q)-1$.
  \end{enumerate}
\end{definition}

\begin{example}\label{ex:ParMat-AP}
  For $\nu = (212,3,23)$ and $\mu=(13,3,222)$ the pair $(A,P)$ defines an element of $\mathrm{ParMat}^{\flat}_{\nu,\mu}$ where

\begin{center}
$A=$
\begin{tabular}{c c | c | c c c}
$0$ & $2$ & $0$ & $0$ & $0$ & $0$\\
$0$ & $0$ & $0$ & $1$ & $0$ & $0$\\
$1$ & $0$ & $0$ & $0$ & $0$ & $1$\\
\hline
$0$ & $0$ & $3$ & $0$ & $0$ & $0$\\
\hline
$0$ & $0$ & $0$ & $1$ & $0$ & $1$\\
$0$ & $1$ & $0$ & $0$ & $2$ & $0$\\
\end{tabular}
and
$P=$
\begin{tabular}{c c | c | c c c}
$\varnothing$ & $\varnothing$ & $\varnothing$ & $\varnothing$ & $\varnothing$ & $\varnothing$\\
$\varnothing$ & $\varnothing$ & $\varnothing$ & $\varnothing$ & $\varnothing$ & $\varnothing$\\
$\varnothing$ & $\varnothing$ & $\varnothing$ & $\varnothing$ & $\varnothing$ & $\varnothing$\\
\hline
$\varnothing$ & $\varnothing$ & $3$ & $\varnothing$ & $\varnothing$ & $\varnothing$\\
\hline
$\varnothing$ & $\varnothing$ & $\varnothing$ & $11$ & $\varnothing$ & $1$\\
$\varnothing$ & $\varnothing$ & $\varnothing$ & $\varnothing$ & $21$ & $\varnothing$\\
\end{tabular}
,\begin{tabular}{c}
$\text{row}(A)=(212,3,23)$\\
$\text{col}(A)=(13,3,222)$
\end{tabular}.
\end{center}
  Later in Example~\ref{ex:BCM-with-B} we will reorganize this pair into a single block matrix, which is more compatible with a componentwise RSK.
\end{example}

\subsection{Flagged block composition matrices}
The partition matrices indexing the cyclotomic web basis are not well suited for direct comparison with tableau combinatorics. Instead,
 we find an equivalent but more combinatorially transparent
model for these data. The
reformulation realizes elements of $\text{ParMat}^\flat_{\nu,\mu}$
as 
$\ell \times \ell$ block composition matrices, whose entries record multiplicities of
biletters subject to explicit length constraints. In this form the data decompose naturally into $\ell$
independent components. This repackaging is the key step in our proof of Conjecture~\ref{conj}.

A \textit{composition matrix} is a matrix whose entries are compositions. More precisely, a composition matrix is a map $A:\NN^2\rightarrow\mathcal{C}$, where $\mathcal{C}$ is the set of all integer compositions, such that all but finitely many entries are the empty composition. We write $A=(a_{ij})$ where $a_{ij}$ is the partition in position $(i,j)$.

\begin{definition}
A \textit{flagged $\ell\times\ell$ block composition matrix} is a collection $\{B^{(pq)}\}_{1\leq p,q\leq\ell}$ of composition matrices $B^{(pq)}=\left(b^{(pq)}_{ij}\right)$ such that for each $p,q$, the entries of the matrix $B^{(pq)}$ satisfy
\[
\text{length}(b_{ij}^{(pq)})\leq\text{min}(p,q),
\]
for all $i,j$. We denote by $\text{BCM}_\ell$ the set of all $\ell\times\ell$ block composition matrices.
\end{definition}
For $B\in\text{BCM}_\ell$, we define the row and column sums to be the $\ell$-multicompositions $\text{row}(B)$, $\text{col}(B)$ defined by
\begin{align*}
\text{row}(B)^{(s)}_t=\sum_{q=1}^\ell\sum_{j}\left|b^{(sq)}_{tj}\right|\qquad\text{and}\qquad
\text{col}(B)^{(s)}_t=\sum_{p=1}^\ell\sum_{i}\left|b^{(ps)}_{it}\right|.
\end{align*}
For $\mu,\nu\vDash_\ell n$ we let
\begin{equation*}
\text{BCM}_\ell(\nu,\mu)\coloneq\left\{B\in\text{BCM}_\ell\setmid \begin{tabular}{c}$\text{row}(B)=\nu$\\$\text{col}(B)=\mu$\end{tabular}\right\}.
\end{equation*}
\begin{example}\label{ex:BCM-with-B}
For $\nu = (212,3,23)$ and $\mu=(13,3,222)$
A flagged $3\times3$ block composition matrix
in $BCM_3(\nu,\mu)$ is

\begin{center}
$B=$
\begin{tabular}{c c | c | c c c}
$\varnothing$ & $2$ & $\varnothing$ & $\varnothing$ & $\varnothing$ & $\varnothing$\\
$\varnothing$ & $\varnothing$ & $\varnothing$ & $1$ & $\varnothing$ & $\varnothing$\\
$1$ & $\varnothing$ & $\varnothing$ & $\varnothing$ & $\varnothing$ & $1$\\
\hline
$\varnothing$ & $\varnothing$ & $03$ & $\varnothing$ & $\varnothing$ & $\varnothing$\\
\hline
$\varnothing$ & $\varnothing$ & $\varnothing$ & $001$ & $\varnothing$ & $01$\\
$\varnothing$ & $1$ & $\varnothing$ & $\varnothing$ & $011$ & $\varnothing$\\
\end{tabular},
\begin{tabular}{c}
$\text{row}(B)=(212,3,23)$\\
$\text{col}(B)=(13,3,222)$
\end{tabular}.
\end{center}
In particular,  $B^{(12)}=B^{(21)}=B^{(23)}=B^{(32)}$ are all the empty composition matrix and
 $B^{(13)}$ is the composition matrix 
\begin{tabular}{c c c}
$\varnothing$ & $\varnothing$ & $\varnothing$\\
$1$ & $\varnothing$ & $\varnothing$ \\
$\varnothing$ & $\varnothing$ & $1$\\
\end{tabular}.
\end{example}

We now explain how flagged block composition matrices recover the original
cyclotomic parameters.

\begin{lemma}\label{lem:parmat-to-bcm}
Fix $\mu,\nu \vDash_\ell n$. There is a bijection
\begin{align*}
\mathrm{ParMat}^{\flat}_{\nu,\mu}&\longleftrightarrow\text{BCM}_\ell(\nu,\mu),\\
(A,P)&\mapsto B,
\end{align*}
satisfying $a_{ij}^{(pq)}=\left|b^{(pq)}_{ij}\right|$ and $\mathrm{length}\left(\eta_{ij}^{(pq)}\right)=\mathrm{length}\left(b_{ij}^{(pq)}\right)-1$.
\end{lemma}
\begin{proof}
There is a natural, length-respecting bijection between compositions of $k$ and partitions with parts at most $k$: if $\lambda$ is such a partition,
the associated composition $c$ is given by
\(
c_t=\lambda_{t-1}-\lambda_t,
\)
with the convention $\lambda_0=k$ and $\lambda_t=0$ for $t>\ell(\lambda)$.
The inverse map recovers $\lambda$ from $c$ by taking partial sums
$\lambda_t=\sum_{s>t} c_s$.

We apply this bijection componentwise on the entries of $(A,P)$. More precisely, the element $(A,P)=\bigl((a_{ij}^{(pq)}),(\eta_{ij}^{(pq)})\bigr)
\in \mathrm{ParMat}^\flat_{\nu,\mu}$
corresponds to $B\in \text{BCM}_\ell(\nu,\mu)$ by setting
\begin{equation*}
\left(b^{(pq)}_{ij}\right)_t=\left(\eta_{ij}^{(pq)}\right)_{t-1}-\left(\eta_{ij}^{(pq)}\right)_t,
\end{equation*}
where we take the convention that $\left(\eta_{ij}^{(pq)}\right)_0\coloneq a_{ij}^{(pq)}$ and that $\left(\eta_{ij}^{(pq)}\right)_k=0$ when $k>\mathrm{length}(\eta_{ij}^{(pq)})$.
\end{proof}

\section{Flagged biwords}
The identification of ${\rm ParMat}$ with flagged block composition matrices allows us to approach defining a level-$\ell$ RSK construction of Conjecture~\ref{conj} by instead bijecting $\mathrm{BCM}_\ell$ with 
pairs of multitableaux. The advantage of this approach is that 
we can extend the familiar bijection between $\NN$-matrices and $\text{BW}(\NN)$ to map flagged block composition matrices onto
 a collection of $\ell${\it-flagged biwords} to which we can then apply classical RSK to obtain pairs of semistandard multitableaux.

\begin{definition}
Let $\mu,\nu\vDash_\ell n$. A \textit{$(\nu,\mu)$-flagged biword} is an $\ell$-tuple $w=(w^{(1)},\omega^{(2)},\ldots,w^{(\ell)})$ of biwords $w^{(i)}\in\text{BW}(\mathcal{S})$ satisfying:
\begin{enumerate}
\item If the letter $a_b$ appears in the biword $w^{(i)}$ then $b\geq i$.
\item The multiset union of top rows $w^{(1)}_\text{top}\sqcup w^{(2)}_\text{top}\sqcup\cdots\sqcup w^{(\ell)}_\text{top}$ forms the $\nu$-alphabet.
\item The multiset union of bottom rows $w^{(1)}_\text{bot}\sqcup w^{(2)}_\text{bot}\sqcup\cdots\sqcup w^{(\ell)}_\text{bot}$ forms the $\mu$-alphabet.
\end{enumerate}
Let $\text{BW}_\ell(\nu,\mu)$ denote the set of all $(\nu,\mu)$-flagged biwords. 
\end{definition}

\begin{example}\label{uvword}
For $\nu=(212,3,23)$ and $\mu=(13,3,222)$, a $(\nu,\mu)$-flagged biword is given by
 $w=(w^{(1)},w^{(2)},w^{(3)})$ where
$$w^{(1)}=\begin{pmatrix}1_1\ 1_1\ 2_1 \ 3_1 \ 3_1\ 2_3 \\ 2_1\ 2_1 \ 1_3 \ 1_1 \ 3_3\ 2_1\end{pmatrix}\,,
w^{(2)}=\begin{pmatrix}1_2\ 1_2\ 1_2 \ 1_3 \ 2_3 \\ 1_2\ 1_2 \ 1_2 \ 3_3\ 2_3\end{pmatrix}\,,\quad\text{and}\quad
w^{(3)}=\begin{pmatrix}1_3\ 2_3\\ 1_3\ 2_3\end{pmatrix}$$   
\end{example}

\begin{lemma}\label{lem:BCM-to-biwords}
Fix $\mu,\nu \models_\ell n$. There is a natural bijection
\[
BCM_\ell(\nu,\mu)\;\longleftrightarrow\; BW_\ell(\nu,\mu),
\]
between flagged block composition matrices and $(\nu,\mu)$--flagged biwords.
\end{lemma}

\begin{proof}
Fix $\mu,\nu\vDash_\ell n$. 
Let $B = \{B^{(pq)}\} \in BCM_\ell(\nu,\mu)$, where
$B^{(pq)} = (b^{(pq)}_{ij})$ is a composition matrix. From $B$, we define $w=(w^{(1)},w^{(2)},\ldots,w^{(\ell)})\in\text{BW}_\ell(\mu,\nu)$ by declaring that
\begin{equation*}
\text{the biletter }\begin{pmatrix}i_p\\j_q\end{pmatrix}\text{ in }w^{(t)} \text{ occurs with multiplicity  }\left(b^{(pq)}_{ij}\right)_t\,,
\end{equation*}
where $\left(b^{(pq)}_{ij}\right)_t$ is the $t^\text{th}$ part of the composition $b^{(pq)}_{ij}$.
In particular, if $w^{(t)}$ is of the form
\begin{equation*}
w^{(t)}=\begin{pmatrix}\alpha_1\ \alpha_2 \ \cdots \ \alpha_s \\ \beta_1 \ \beta_2 \ \cdots \ \beta_s\end{pmatrix}
,\end{equation*}
for some $s\geq 0$, then for each $1\leq p,q\leq\ell$ and $i,j\geq 1$ we have
\begin{equation*}
\#\left\{k\in[s]\setmid\begin{pmatrix}\alpha_k\\\beta_k\end{pmatrix}=\begin{pmatrix}i_p\\j_q\end{pmatrix}\right\}=\left(b^{(pq)}_{ij}\right)_t.
\end{equation*}
The length condition $\ell(b^{(pq)}_{ij}) \le \min(p,q)$ ensures that
$w^{(t)}$ satisfies the flagging condition, and the row and column sum
conditions imply that the multiset unions of the top and bottom rows of
$\{w^{(t)}\}_{t=1}^\ell$ form the $\mu$-- and $\nu$--alphabets, respectively.
Thus $w = (w^{(1)},\dots,w^{(\ell)})$ is a $(\nu,\mu)$--flagged biword.

It is straightforward to reconstruct $B$ from $w$ by setting each 
\begin{equation*}
\left(b^{(pq)}_{ij}\right)_t=\#\text{ of }\begin{pmatrix}i_p\\j_q\end{pmatrix}\text{ appearing in }\omega^{(t)}\,.
\end{equation*}
\end{proof}

\begin{example}
For  $\ell=3$,  $\mu=(212,3,23)$ and $\nu=(13,3,222)$, the flagged compositional
block matrix $B$ from Example~\ref{ex:BCM-with-B}, written with 
a letter-coordinate key, is sent to the flagged multi biword of 
Example~\ref{uvword}\\
\begin{tabular}{c | c c | c | c c c}
\color{violet}\text{key:}& \color{violet}$1_1$ & \color{violet}$2_1$ & \color{violet}$1_2$ & \color{violet}$1_3$ & \color{violet}$2_3$ & \color{violet}$3_3$\\
\hline
{\color{violet}$1_1$}&$\varnothing$ & $2$ & $\varnothing$ & $\varnothing$ & $\varnothing$ & $\varnothing$\\
{\color{violet}$2_1$}&$\varnothing$ & $\varnothing$ & $\varnothing$ & $1$ & $\varnothing$ & $\varnothing$\\
{\color{violet}$3_1$}&$1$ & $\varnothing$ & $\varnothing$ & $\varnothing$ & $\varnothing$ & $1$\\
\hline
{\color{violet}$1_2$}&$\varnothing$ & $\varnothing$ & $03$ & $\varnothing$ & $\varnothing$ & $\varnothing$\\
\hline
{\color{violet}$1_3$}&$\varnothing$ & $\varnothing$ & $\varnothing$ & $001$ & $\varnothing$ & $01$\\
{\color{violet}$2_3$}&$\varnothing$ & $1$ & $\varnothing$ & $\varnothing$ & $011$ & $\varnothing$\\
\end{tabular}
$\longleftrightarrow$
$\left(\begin{pmatrix}1_1\ 1_1\ 2_1 \ 3_1 \ 3_1\ 2_3 \\ 2_1\ 2_1 \ 1_3 \ 1_1 \ 3_3\ 2_1\end{pmatrix},
\begin{pmatrix}1_2\ 1_2\ 1_2 \ 1_3 \ 2_3 \\ 1_2\ 1_2 \ 1_2 \ 3_3\ 2_3\end{pmatrix}\,,
\begin{pmatrix}1_3\ 2_3\\ 1_3\ 2_3\end{pmatrix}\right)$.
\end{example}
\begin{lemma}\label{lem:biwords-to-tableaux}
Fix $\mu,\nu \models_\ell n$. There is a bijection
\[
BW_\ell(\nu,\mu)\;
\longleftrightarrow\coprod_{\lambda\vdash_\ell n}\text{\rm SST}_\ell(\lambda,\mu)\times\text{\rm SST}_\ell(\lambda,\nu),
\]
obtained by applying the classical Robinson--Schensted--Knuth correspondence
componentwise, after an order-preserving relabeling of the $\mu$-- and
$\nu$--alphabets.
\end{lemma}
\begin{proof}
Let $w=(w^{(1)},\dots,w^{(\ell)})\in BW_\ell(\nu,\mu)$.
Construct the pair $(P,Q)$ of semistandard multitableaux componentwise using the usual RSK correspondence of Thm~\ref{classicalRSK}, where
we set $(P^{(i)},Q^{(i)})=\text{RSK}(w^{(i)})$, but with the modification that  the (totally ordered) $\mu$- and $\nu$-alphabets serve in place of the usual letters $[n]$.

Precisely, let $U_\mu,U_\nu\subset S$ denote the $\mu$-- and $\nu$--alphabets.
Choose order-preserving injections
\[
\phi_\mu: U_\mu \hookrightarrow \mathbb{N},
\qquad
\phi_\nu: U_\nu \hookrightarrow \mathbb{N},
\]
for example by mapping the $k$th smallest element of each alphabet to $k$.
For a biword
\[
\pi=\begin{pmatrix}\alpha_1&\cdots&\alpha_s\\ \beta_1&\cdots&\beta_s\end{pmatrix}
\in BW(S),
\]
define $\phi(\pi)\in BW(\mathbb{N})$ by applying $\phi_\nu$ to the top row and
$\phi_\mu$ to the bottom row.
Then, for each $1\le t\le\ell$, apply the classical RSK correspondence to the relabeled
biword $\phi(w^{(t)})$ to obtain a pair of tableaux
\[
(\widetilde P(t),\widetilde Q(t)),
\]
of common shape with entries in $\mathbb{N}$.
Finally,
 replace each entry of $\widetilde P(t)$ and $\widetilde Q(t)$ by its unique preimage
under $\phi_\mu$ and $\phi_\nu$, respectively, producing tableaux
$P(t)$ and $Q(t)$ with entries in the original alphabets.

The flagging condition on $w$ guarantees that entries of $P(t)$ and $Q(t)$ satisfy
the level restrictions, while the content conditions ensure that the collections
$\{P(t)\}$ and $\{Q(t)\}$ assemble into semistandard multitableaux of contents
$\mu$ and $\nu$. Since classical RSK depends only on the relative order of letters,
the resulting multitableaux are independent of the chosen relabeling.

Invertibility follows from invertibility of classical RSK together with the
bijectivity of the relabeling maps.
\end{proof}

\begin{example}
Taking individually the biword components $w^{(1)}$, $w^{(2)}$, and $w^{(3)}$ from Example \ref{uvword}, we apply a relabled version of classical RSK 
$$w^{(1)}=\begin{pmatrix}1_1\ 1_1\ 2_1 \ 3_1 \ 3_1\ 2_3 \\ 2_1\ 2_1 \ 1_3 \ 1_1 \ 3_3\ 2_1\end{pmatrix}\mapsto\left(
\begin{ytableau}1_1&2_1&2_1&3_3\\2_1&1_3\end{ytableau}\,\mathhuger{,}\, \begin{ytableau}1_1&1_1&2_1&3_1\\3_1&2_3\end{ytableau}\right)=(P^{(1)},Q^{(1)}),$$
$$w^{(2)}=\begin{pmatrix}1_2\ 1_2\ 1_2 \ 1_3 \ 2_3 \\ 1_2\ 1_2 \ 1_2 \ 3_3\ 2_3\end{pmatrix}\mapsto\left(
\begin{ytableau}1_2&1_2&1_2&2_3\\3_3\end{ytableau}\,\mathhuger{,}\, \begin{ytableau}1_2&1_2&1_2&1_3\\2_3\end{ytableau}\right)=(P^{(2)},Q^{(2)}),$$
$$w^{(3)}=\begin{pmatrix}1_3\ 2_3\\ 1_3\ 2_3\end{pmatrix}\mapsto\left(
\begin{ytableau}1_3&2_3\end{ytableau}\,\mathhuger{,}\, \begin{ytableau}1_3&2_3\end{ytableau}\right)=(P^{(3)},Q^{(3)}).$$
The relabling involves temporarily changing the alphabet to $\NN$, for example with $w^{(2)}$

\begin{center}
\resizebox{\linewidth}{!}{$
\begin{pmatrix}1_2\ 1_2\ 1_2 \ 1_3 \ 2_3 \\ 1_2\ 1_2 \ 1_2 \ 3_3\ 2_3\end{pmatrix}
\mapsto
\begin{pmatrix}1\ 1\ 1 \ 2 \ 3 \\ 1\ 1 \ 1 \ 4\ 3\end{pmatrix}
\mapsto
\left(
\begin{ytableau}1&1&1&3\\4\end{ytableau}\,\mathhuger{,}\,
\begin{ytableau}1&1&1&2\\3\end{ytableau}
\right)
\mapsto
\left(
\begin{ytableau}1_2&1_2&1_2&2_3\\3_3\end{ytableau}\,\mathhuger{,}\,
\begin{ytableau}1_2&1_2&1_2&1_3\\2_3\end{ytableau}
\right)
$}.
\end{center}

\end{example}

\section{The correspondence}
\begin{theorem}
Fix $\mu,\nu \models_\ell n$. Applying the classical RSK correspondence of Thm~\ref{classicalRSK}
componentwise defines a bijection
\begin{equation}
\text{ParMat}^\flat_{\nu,\mu}
\longleftrightarrow\coprod_{\lambda\vdash_\ell n}\text{\rm SST}_\ell(\lambda,\mu)\times\text{\rm SST}_\ell(\lambda,\nu).
\end{equation}
When restricted to multipartitions satisfying
$\mu^{(\ell)}=\nu^{(\ell)}=\lambda^{(\ell)}=\varnothing$, this bijection agrees with the
level--$(\ell-1)$ correspondence. For $\ell=1$, it reduces to the usual RSK
correspondence.
\end{theorem}
\begin{proof}
We first use Lemma~\ref{lem:parmat-to-bcm} to replace the domain with 
$\text{\rm BCM}_\ell(\nu,\mu)$. From there,
Lemma~\ref{lem:BCM-to-biwords}  takes flagged block composition matrices bijectively to
flagged biwords. Lemma~\ref{lem:biwords-to-tableaux} then applies
classical RSK componentwise to obtain the desired correspondence. The
specialization and restriction properties follow directly from the definitions.
\end{proof}

\begin{example}
For a chosen $(P,Q)=\left((P^{(1)},P^{(2)},P^{(3)})\mathlarger{,}(Q^{(1)},Q^{(2)},Q^{(3)})\right)$, the inverse map can be explicitly realized using the classical inverse RSK.
\begin{equation*}
(P^{(1)},Q^{(1)})=
\left(
\begin{ytableau}1_1&1_1&2_1\\2_1&2_1\end{ytableau}\,\mathhuger{,}\, \begin{ytableau}1_1&2_1&2_1\\2_1&1_2\end{ytableau}\right)
\mapsto
\begin{pmatrix}1_1\ 2_1\ 2_1 \ 2_1 \ 1_2 \\ 2_1\ 1_1 \ 2_1 \ 2_1\ 1_1\end{pmatrix}=w^{(1)},
\end{equation*}
\begin{equation*}
(P^{(2)},Q^{(2)})=
\left(
\begin{ytableau}1_2&2_3\end{ytableau}\,\mathhuger{,}\, \begin{ytableau}2_2&1_3\end{ytableau}\right)
\mapsto
\begin{pmatrix}2_2\ 1_3\\ 1_2\ 2_3\end{pmatrix}=w^{(2)},
\end{equation*}
\begin{equation*}
(P^{(3)},Q^{(3)})=
\left(
\begin{ytableau}1_3&1_3\\3_3\end{ytableau}\,\mathhuger{,}\, \begin{ytableau}1_3&2_3\\2_3\end{ytableau}\right)
\mapsto
\begin{pmatrix}1_3\ 2_3\ 2_3\\ 3_3\ 1_3 \ 1_3\end{pmatrix}=w^{(3)}.
\end{equation*}
Then $w=(w^{(1)},w^{(2)},w^{(3)})$ is assembled into $B=$
\begin{tabular}{c c | c | c c c}
$\varnothing$ & $1$ & $\varnothing$ & $\varnothing$ & $\varnothing$ & $\varnothing$\\
$1$ & $2$ & $\varnothing$ & $\varnothing$ & $\varnothing$ & $\varnothing$\\
\hline
$1$ & $\varnothing$ & $\varnothing$ & $\varnothing$ & $\varnothing$ & $\varnothing$\\
$\varnothing$ & $\varnothing$ & $01$ & $\varnothing$ & $\varnothing$ & $\varnothing$\\
\hline
$\varnothing$ & $\varnothing$ & $\varnothing$ & $\varnothing$ & $01$ & $001$\\
$\varnothing$ & $\varnothing$ & $\varnothing$ & $002$ & $\varnothing$ & $\varnothing$\\
\end{tabular}.
\end{example}

\bibliographystyle{plain}
\bibliography{mycitations}

\end{document}